\documentclass{amsart}
\usepackage{amsmath}
\usepackage{amsthm}
\usepackage{amsfonts,mathrsfs}
\usepackage{amssymb}
\usepackage{graphicx}
\usepackage{enumitem}
\usepackage{color}
\usepackage{verbatim}
\usepackage{enumerate}
\usepackage{amssymb,latexsym}
\usepackage{cancel}
\usepackage{cite}
\usepackage{lineno}
\usepackage{hyperref}
\usepackage{amscd}
\usepackage{lineno}

\theoremstyle{plain}
\newtheorem{theorem}{Theorem}[section]
\newtheorem{lemma}[theorem]{Lemma}
\newtheorem{remark}[theorem]{Remark}

\theoremstyle{definition}

\newtheoremstyle{TheoremNum}
	{\topsep}{\topsep}              
  {\itshape}                      
  {}                              
  {\bfseries}                     
  {.}                             
  { }                             
  {\thmname{#1}\thmnote{ \bfseries #3}}

\newcommand{\F}{\mathbb F}
\newcommand{\fqn}{\mathbb{F}_{q^n}}
\newcommand{\fq}{\mathbb{F}_{q}}

\newcommand{\Z}{\mathbb Z}

\newcommand{\PG}{\mathrm{PG}}

\newcommand{\cC}{\mathcal C}

\newcommand{\Aut}{\mathrm{Aut}}
\newcommand{\PGL}{\mathrm{PGL}}
\newcommand{\ba}{\mathbf{a}}
\newcommand{\bb}{\mathbf{b}}
\newcommand{\PGamL}{\mathrm{P\Gamma L}}

\newcommand{\Tr}{ \ensuremath{ \mathrm{Tr}}}

\newcommand{\RN}[1]{%
  \textup{\uppercase\expandafter{\romannumeral#1}}%
}

 \def\zhou#1 {\fbox {\footnote {\ }}\ \footnotetext { From Yue: {\color{blue}#1}}}

 \def\chen#1 {\fbox {\footnote {\ }}\ \footnotetext { From Tang: {\color{red}#1}}}

\begin{document}
	\title[On the automorphism group of LP scattered linear sets]{On the automorphism groups of Lunardon-Polverino scattered linear sets}
	\author[W. Tang]{Wei Tang\textsuperscript{\,1}}
	\author[Y. Zhou]{Yue Zhou\textsuperscript{\,1,$\dagger$}}
	\author[F. Zullo]{Ferdinando Zullo\textsuperscript{\,2}}
	\address{\textsuperscript{1}Department of Mathematics, National University of Defense Technology, 410073 Changsha, China}	\email{yue.zhou.ovgu@gmail.com}
	\address{\textsuperscript{2}Department of Mathematics and Physics, University of Campania ``Luigi Vanvitelli'', 81012 Caserta, Italy}
	\email{ferdinando.zullo@unicampania.it}
	\address{\textsuperscript{$\dagger$}Corresponding author}
	\date{\today}
	\begin{abstract}
    Lunardon and Polverino introduced in 2001 a new family of maximum scattered linear sets in $\mathrm{PG}(1,q^n)$ to construct linear minimal R\'edei blocking sets. This family has been extended first by Lavrauw, Marino, Trombetti and Polverino in 2015 and then by Sheekey in 2016 in two different contexts (semifields and rank metric codes).  These linear sets are called \emph{Lunardon-Polverino} linear sets and this paper aims to determine their automorphism groups, to solve the equivalence issue among Lunardon-Polverino linear sets and to establish the number of inequivalent linear sets of this family. We then elaborate on this number, providing explicit bounds and determining its asymptotics.
	\end{abstract}
	\maketitle
	
\section{Introduction}

The determination of the automorphism group of combinatorial objects has been a central topic in combinatorics for many years, especially because it gives information on the \emph{symmetry} of the considered structure. However, very little is known about the automorphism groups of linear sets.

Let $q=p^r$, for some prime $p$ and a positive integer $r$.
Let $n$ be a positive integer and let $V$ be a $2$-dimensional vector space over $\fqn$ (up to coordinatize, we can suppose that $V=\F_{q^n}^2$). A point set $L$ of $\Lambda=\PG(V,\F_{q^n})=\PG(1,q^n)$ is said to be an \emph{$\F_q$-linear set} of $\Lambda$ of rank $k$ if it is defined by the non-zero vectors of a $k$-dimensional $\F_q$-vector subspace $U$ of $V$, i.e.
\[L=L_U:=\{\langle {\bf u} \rangle_{\mathbb{F}_{q^n}} : {\bf u}\in U\setminus \{{\bf 0} \}\}.\]
For any point $P=\PG(Z,\mathbb{F}_{q^n})$ of $\Lambda$, the \emph{weight} of $P$ in $L_U$ is defined as $w_{L_U}(P)=\dim_{\mathbb{F}_q}(U\cap Z)$.
For an $\fq$-linear set $L$ of rank $k$ the following holds
\begin{equation}\label{eq:size} 
|L|\leq \frac{q^{k}-1}{q-1}. 
\end{equation}
We also recall that two linear sets $L_U$ and $L_W$ in the projective line $\Lambda$ are said to be \emph{$\mathrm{P\Gamma L}$-equivalent} (or simply \emph{equivalent}) if there is an element $\varphi$ in $\mathrm{P\Gamma L}(2,q^n)$ such that $L_U^{\varphi} = L_W$.
In general, it is very challenging to determine whether or not two linear sets are equivalent; see e.g.\ \cite{csajbok_classes_2018}.

An important family of linear sets are the \emph{scattered}  ones, that is those linear sets satisfying the equality in \eqref{eq:size} (or, equivalently, the linear sets whose points have weight one); see e.g.\ \cite{Polverino}.
Scattered linear sets were originally introduced by Blokhuis and Lavrauw in \cite{BL2000} in a more general setting.  Since the seminal paper by Blokhuis and Lavrauw, linear sets attracted a lot of attention especially because of their connections with blocking sets, semifields, rank metric codes, etc (see e.g. \cite{LVdV2015,Polverino,John}).
In \cite{BL2000}, it has also been proved that a scattered linear set in $\Lambda$ can have rank at most $n$ and a scattered linear set of rank $n$ is said to be \emph{maximum scattered}.
The first example of maximum scattered linear set is the linear set \emph{of pseudoregulus type}, that is 
\[ \{ \langle (x,x^{q^s}) \rangle_{\fqn} \colon x \in \fqn^* \}, \]
with $\gcd(s,n)=1$.
In particular,
\[ \left\{ \langle (x,x^{q}) \rangle_{\fqn} \colon x \in \fqn^* \right\}=\left\{ \langle (x,x^{q^s}) \rangle_{\fqn} \colon x \in \fqn^* \right\}, \]
for any $s$ such that $\gcd(s,n)=1$.
This family has been studied in details, indeed its geometric characterization has been provided by Csajb\'ok and Zanella in \cite{CsZpseudo}. 

The next example of maximum scattered linear set in $\Lambda$ which has been given by Lunardon and Polverino in \cite{LP} and generalized in \cite{LMPT,John} is
\[ L=\{ \langle (x,x^{q^s}+\delta x^{q^{n-s}}) \rangle_{\fqn} \colon x \in \fqn^* \}, \]
with $n\geq 4$, $\gcd(s,n)=1$, $q\neq 2$, and nonzero $\delta$ such that $N_{q^n/q}(\delta)=\delta^{\frac{q^n-1}{q-1}}\ne 1$. This example is known as \emph{Lunardon-Polverino} linear set. If we set $\delta=0$,  then the Lunardon-Polverino linear set reduces to be of pseudoregulus type.
In \cite[Theorem 3.4]{Zanella}, Zanella proved that the condition $N_{q^n/q}(\delta)\ne 1$ is necessary for $L$ to be scattered (see also \cite{BMZZ}).
Also, a more geometric description for such linear sets has been given in \cite{ZanellaZ} and an upper bound on the number of inequivalent Lunardon-Polverino linear sets has been provided in \cite[Proposition 2.3]{LMTZscatt}, but the exact number is unknown. 

Let $L_U$ be an $\fq$-linear set of rank $n$ in $\Lambda$. It is well-known that ${\rm P\Gamma L}(2,q^n)$ is $3$-transitive on $\Lambda$, and so we may assume that $L_U$ does not contain the point $\langle(0,1)\rangle_{\fqn}$, so that $L_U$ can be written as
\[ L_f=\{ \langle (x,f(x)) \rangle_{\fqn} \colon x \in \fqn^* \}, \]
where $f$ is a \emph{$q$-polynomial} over $\fqn$, that is $f=\sum_{i=0}^{n-1}a_iX^{q^i}\in \fqn[X]$.
Sheekey in \cite{John} called $f$ \emph{scattered} if $L_f$ turns out to be a scattered $\fq$-linear set.
When choosing $f=X^{q^s}$ then $L_f$ turns out to be the linear set of pseudoregulus type, whereas if $f=X^{q^s}+\delta X^{q^{n-s}}$ then $L_f$ is the Lunardon-Polverino linear set.
These two families of polynomials play an important role in the classifications of scattered polynomials which are \emph{exceptional}, that is, roughly speaking, polynomials which are scattered for infinitely many $n$. Indeed, the known classification results suggest that these two families are the only examples of exceptional scattered polynomials; see \cite{BZ,BM,BZZ,FM}.

In this paper we study Lunardon-Polverino linear sets. 
We first completely solve the equivalence issue among two Lunardon-Polverino linear sets and we also determine the automorphism group of a Lunardon-Polverino linear set. This has been done by exploiting the problem of determining the relations between the coefficients of two $q$-polynomials $f$ and $g$ for which the value sets of $f(x)/x$ and $g(x)/x$ coincide.
As a byproduct and making use of careful computations, we completely determine $\Lambda(n,q)$ the number of inequivalent Lunardon-Polverino linear sets.
Finally we compute explicit upper and lower bounds on $\Lambda(n,q)$ and we determine its asymptotic behavior as 
$r \to +\infty$,  by making use of a result by Gronwall (1913) involving the Euler-Mascheroni constant.

\section{Auxiliary results}

We need the following well known result on the greatest common divisors. A proof can be found in \cite{Payne}.
\begin{lemma}\label{le:gcd}
	Let $v(b):= \max \{ i: 2^i \mid b\}$. For  a prime $p$,
	\begin{equation}\label{eq:gcd}
		\gcd(p^i+1,p^j-1)=
		\begin{cases}
			p^{\gcd(i,j)}+1, &v(i)<v(j); \\ 
			\frac{(-1)^{p+1}+3}{2},&\text{otherwise}.
		\end{cases}
	\end{equation}
\end{lemma}

Let $s$ be a positive integer such that $\gcd(s,n)=1$ and let $\sigma \colon x \in \fqn \mapsto x^{q^s} \in \fqn$.
A $\sigma$-\emph{polynomial} (or a \emph{linearized polynomial}) over $\F_{q^n}$ is a polynomial of the following form
\[f=\sum_{i=0}^{\ell} a_i X^{\sigma^i},\]
where $a_i\in \F_{q^n}$ and $\ell$ is a positive integer.
We will also call $f$ a $q^s$-\emph{polynomial}.
If $a_\ell \neq 0$ we say that $\ell$ is the $\sigma$-\emph{degree} of $f$.
We will denote by $\mathcal{L}_{n,\sigma}$ the set of all $\sigma$-polynomials over $\F_{q^n}$ with $\sigma$-degree less than $n$ (when $s=1$, we will denote it by $\mathcal{L}_{n,q}$). The set $\mathcal{L}_{n,\sigma}$ equipped with the classical sum of polynomials, the composition modulo $X^{\sigma^n}-X$ and the scalar multiplication by an element in $\fq$, is an $\fq$-algebra isomorphic to the algebra of $\fq$-linear endomorphism of $\fqn$.
This isomorphism allows us to identify a $\sigma$-polynomial with the $\fq$-linear map it defines.
We also want to underline that different choices of $\sigma$ give a different  representation of an element, but leave unchanged the algebra $\mathcal{L}_{n,\sigma}$. Hence, $\mathcal{L}_{n,\sigma}$ is the same for any choice of $\sigma$, and one can just consider $\mathcal{L}_{n,q}$.
For more details on linearized polynomials we refer to \cite[Chapter 3, Section 4]{lidl_finite_1997}.

Consider the non-degenerate symmetric bilinear form of $\F_{q^n}$ over $\F_q$ defined  by
\begin{equation*}
\langle x,y\rangle= \mathrm{Tr}_{q^n/q}(xy),
\end{equation*}
for every $x,y \in \F_{q^n}$  where $\mathrm{Tr}_{q^n/q}(x)=x+\ldots+x^{q^{n-1}}$.
The \emph{adjoint} $\hat{f}$ of the $\sigma$-polynomial $\displaystyle f=\sum_{i=0}^{n-1} a_iX^{\sigma^i} \in \mathcal{L}_{n,\sigma}$ with respect to the bilinear form $\langle\cdot,\cdot\rangle$, i.e. the unique function over $\fqn$ satisfying
\[ \mathrm{Tr}_{q^n/q}(yf(z))=\mathrm{Tr}_{q^n/q}(z\hat{f}(y)) \]
for every $y,z \in \F_{q^n}$, is given by
\begin{equation}\label{eq:adj} \hat{f}=\sum_{i=0}^{n-1} a_i^{\sigma^{n-i}}X^{\sigma^{n-i}}.
\end{equation}

In \cite[Lemma 2.6]{BGMP2015} it has been proved that a $\sigma$-polynomial $f$ and its adjoint always define the same linear set, and following the proof of \cite[Lemma 2.6]{BGMP2015} one also obtain that the weight distribution of $L_f$ coincides with the weight distribution of $L_{\hat{f}}$; see also \cite[Theorem 1]{McGS}. 
More precisely, the following holds (which we prove for sake of completeness in a different way).

\begin{lemma}\label{le:LS_adjoint}
	Let $f$ be a $\sigma$-polynomial in $\mathcal{L}_{n,\sigma}$. Then 
	\begin{equation}\label{eq:LS_adjoint}
		\#\left\{x\in\F_{q^n}^*  : \frac{f(x)}{x}=b \right\} = \#\left\{y\in\F_{q^n}^*  : \frac{\hat{f}(y)}{y}=b \right\}.
	\end{equation}
	In particular, $L_f=L_{\hat{f}}$ and $w_{L_f}(P)=w_{L_{\hat{f}}}(P)$ for every point $P \in \mathrm{PG}(1,q^n)$. 
\end{lemma}
\begin{proof}
	Let $\chi(x)=\zeta_p^{\Tr_{q^n/p}(x)  }$, where $p$ is the characteristic of $\F_{q^n}$ and $\zeta_p$ is a $p$-th root of unity in $\mathbb{C}$. Define $\chi_a(x)=\chi(ax)$ for $a\in \F_{q^n}$. Then the character group of the additive group of $\F_{q^n}$ is $\{ \chi_a: a\in \F_{q^n} \}$. For characters of groups, we refer to \cite[Chapter 8]{JLL}. 
	
	
	By the definition of adjoints, for any $b\in \F_{q^n}$,
	\begin{equation}\label{eq:LS_adjoint_character}
		\sum_{y\in \F_{q^n}}\sum_{x\in \F_{q^n}} \chi \left( (f(x)-bx)y\right)=\sum_{x\in \F_{q^n}}\sum_{y\in \F_{q^n}} \chi \left( (\hat{f}(y)-by)x\right).
	\end{equation}
	By the orthogonality of characters, the left-hand-side of \eqref{eq:LS_adjoint_character} equals
	\[
		\#\{x\in\F_{q^n}  : f(x)=bx \} \cdot q^n=\left(\#\left\{x\in\F_{q^n}^*  : \frac{f(x)}{x}=b \right\}+1\right) \cdot q^n.
	\]
	The same result holds for the right-hand-side of \eqref{eq:LS_adjoint_character}. Therefore \eqref{eq:LS_adjoint} is proved and the last part of the assertion is a direct consequence.
\end{proof}

Given two arbitrary linearized polynomials $f$ and $g$, it is usually quite difficult to determine whether $L_f=L_g$. However, we can obtain some conditions involving the coefficients of both $f$ and $g$. Indeed, in \cite[Lemma 3.4]{csajbok_classes_2018} the following condition arises
\[
\sum_{x\in\F_{q^n}} \left(\frac{f(x)}{x}\right)^e=\sum_{x\in\F_{q^n}} \left(\frac{g(x)}{x}\right)^e,
\]
for any positive integer $e$.
For some special values of $e$, Csajb\'ok, Marino and Polverino \cite{csajbok_classes_2018} have provided some more explicit conditions.

\begin{lemma}\label{le:jcta2018}\cite[Lemma 3.6]{csajbok_classes_2018}
	Let $f=\sum_{i=0}^{n-1} \alpha_i X^{q^i}$ and $g=\sum_{i=0}^{n-1} \beta_i X^{q^i}$ be two $q$-polynomials over $\F_{q^n}$ such that $L_f=L_g$. Then 
	\begin{equation}\label{eq:jcta_1}
		\alpha_0=\beta_0
	\end{equation}
	and 
	\begin{equation}\label{eq:jcta_2}
		 \alpha_k \alpha_{n-k} ^{q^k} = \beta_k \beta_{n-k}^{q^k} 
	\end{equation}
	for $k\in\{1,2,\cdots, n-1\}$, and 
	\[ \alpha_1 \alpha_{k-1}^q\alpha_{n-k}^{q^k}  + \alpha_k \alpha_{n-1}^{q} \alpha_{n-k+1}^{q^k}=   \beta_1 \beta_{k-1}^q\beta_{n-k}^{q^k}  + \beta_k \beta_{n-1}^{q} \beta_{n-k+1}^{q^k},\]
	for $k\in \{2,3,\cdots, n-1\}$.
\end{lemma}

For some linearized polynomials with two terms, we can obtain a further necessary condition.

\begin{lemma}\label{le:two_terms}
	Let $s$ and $n$ be two relatively prime positive integers. Let $\sigma\colon x \in \fqn\mapsto x^{q^s}\in \fqn$. For any nonzero elements $\theta, \delta,d\in \F_{q^n}$, let  $f=X^{\sigma}+\theta X^{\sigma^{n-1}}$ and $g=d(X^{\sigma}+\delta X^{\sigma^{n-1}})$ be two $q$-polynomials over $\F_{q^n}$ such that $L_f=L_g$. Then
	\begin{equation}\label{eq:two_terms_odd}
		1+N_{q^n/q}(\theta)=N_{q^n/q}(d)(1+N_{q^n/q}(\delta)),
	\end{equation}
	when $n$ is odd, and
	\begin{equation}\label{eq:two_terms_even}
	1+N_{q^n/q}(\theta)+N_{q^n/q^2}(\theta)+N_{q^n/q^2}(\theta)^\sigma=N_{q^n/q}(d)(	1+N_{q^n/q}(\delta)+N_{q^n/q^2}(\delta)+N_{q^n/q^2}(\delta)^\sigma),
	\end{equation}
	when $n$ is even.
\end{lemma}
\begin{proof}
	The $(1+q+q^2+\cdots+q^{n-1})$-th power of $f(X)/X$ modulo $X^{\sigma^n}-X$ equals,
	\begin{equation}\label{eq:norm-th_power}
			(X^{\sigma-1}+\theta X^{\sigma^{n-1}-1})(X^{\sigma^2-\sigma}+\theta^\sigma X^{\sigma^{n}-\sigma})\cdots(X^{\sigma^{n}-\sigma^{n-1}}+\theta^{\sigma^{n-1}} X^{\sigma^{n-2}-\sigma^{n-1}}).
	\end{equation}
	As $\gcd(s,n)=1$,  the coefficient of $X^{q^n-1}$ in  \eqref{eq:norm-th_power} modulo $X^{q^n}-X$ is the same as the coefficient of $X^{\sigma^n-1}$ modulo $X^{\sigma^n}-X$. 
	
	To get $X^{\sigma^n-1}$ in \eqref{eq:norm-th_power}, we only have to find all the possible ways to choose exactly one term $Z_i$ from each set $\{X^{\sigma^{i+1}-\sigma^i},  X^{\sigma^{i-1}-\sigma^i}  \}$ such that $\prod_{i=0}^{n-1}Z_i\equiv X^{\sigma^n-1} \pmod{X^{\sigma^n}-X}$. To determine it, we may consider the two terms $X^{\sigma^{i+1}-\sigma ^i}$ and $X^{\sigma^{i-1}-\sigma^i}$ in each $(X^{\sigma^{i+1}-\sigma^i}+\theta^\sigma X^{\sigma^{i-1}-\sigma^i})$ as two arrays of length $n$ defined recursively:
	\begin{align*}
	\mathbf{a}_0&=(-1, 1, 0,\cdots,0)\in \mathbb{Z}^{n}, \\
	\mathbf{b}_0&=(-1, 0,\cdots,0,1) \in \mathbb{Z}^{n},
	\end{align*}
	and for $i\in\{1,2,\cdots,n-1\}$,  $\mathbf{a}_i$ (resp. $\mathbf{b}_i$) is obtained from $\mathbf{a}_{i-1}$ (resp. $\mathbf{b}_{i-1}$) by applying a shift on the right on its entries, that is
	\begin{align*}
	\mathbf{a}_i&=(\overbrace{0,\cdots, 0}^{i}, -1, 1, 0,\cdots,0)\in \mathbb{Z}^{n}, \\
	\mathbf{b}_i&=(\overbrace{0,\cdots, 0}^{i-1}, 1, -1, 0,\cdots,0) \in \mathbb{Z}^{n}.
	\end{align*}
	In this way, $\mathbf{a}_i$ represents the exponents of $\sigma$ in the term $X^{\sigma^{i+1}-\sigma ^i}$ and, similarly, $\mathbf{b}_i$ represents the exponents of $\sigma$ in the term $X^{\sigma^{i-1}-\sigma ^i}$, for any $i \in \{0,\ldots,n-1\}$.
	Since our aim is to find the coefficient of $X^{\sigma^n-1}$, we have to find all the possible sum in which every addend is exactly one element from $\mathbf{a}_i$ and $\mathbf{b}_i$ for each $i\in \{0,\cdots, n-1\}$ and such that the sum equals to the zero vector $\mathbf{0}$, which represents $X^{\sigma^n-1}$.
	For any $n$, it is easy to see that 
	\begin{equation}\label{eq:two_terms_1}
		\sum_{i=0}^{n-1}\ba_i=\sum_{i=0}^{n-1}\bb_i=\mathbf{0}.
	\end{equation}
	When $n$ is even, we also have
	\begin{equation}\label{eq:two_terms_2}
		\sum_{2\mid i}\ba_i+\sum_{2\nmid j}\bb_j=\sum_{2\nmid i}\ba_i+\sum_{2\mid j}\bb_j=\mathbf{0}.
	\end{equation}
	
	\medskip
	\noindent\textbf{Claim}:  All the possible ways to get $\mathbf{0}$ are given in \eqref{eq:two_terms_1} and \eqref{eq:two_terms_2}. 
	\medskip
	
	To prove the claim, we first have to notice that the two nonzero elements in $\ba_i$ or $\bb_i$ are adjacent to each other. Consequently, if $\ba_i$ appears in the sum,  to make the $(i+2)$-th position in the array $0$, we have to add exactly one element from $\{\ba_{i+1}, \bb_{i+1}\}$. Moreover, if 
	\[\ba_i+\ba_{i+1}=(\overbrace{0,\cdots, 0}^{i}, -1, 0,1, 0,\cdots,)\]
	is contained in the sum, then it is easy to see that the $(i-1)$-th term of the sum has to be $\ba_{i-1}$. The same result also holds for $\bb_i+\bb_{i+1}$. 
	Whereas, if $\ba_i+\bb_{i+1}$ is contained in the sum, since $\ba_i+\bb_{i+1}=\mathbf{0}$, then we necessarily obtain the case \eqref{eq:two_terms_2}.
	As a consequence, we complete the proof of the claim.
	
	By \textbf{Claim}, the coefficient of $X^{q^n-1}$ in \eqref{eq:norm-th_power} is
	\[c_{q^n-1}=
	\begin{cases}
		1+N_{q^n/q}(\theta), & \text{ if }2\nmid n;\\
				1+N_{q^n/q}(\theta)+N_{q^n/q^2}(\theta)+N_{q^n/q^2}(\theta)^\sigma, & \text{ if }2\mid n.
	\end{cases}
	\]
	By the well known equation 
	\[
	\sum_{x\in\F_{q^n}} x^e = 	\begin{cases}
	-1, & q^n-1 \mid e;\\
	0,& \text{otherwise},
	\end{cases}
	\]
	we have
	\[
	\sum_{x\in \F_{q^n}} (x^{q^s-1}+\theta x^{q^{n-s}-1})^{1+q+\cdots+q^{n-1}}= c_{q^n-1}\sum_{x\in\F_{q^n}}x^{q^n-1}=-c_{q^n-1}.
	\]
	As $L_g=L_f$, by comparing the coefficients of $X^{q^n-1}$ of $(f(X)/X)^{1+q+\cdots+q^{n-1}}$ and $(g(X)/X)^{1+q+\cdots+q^{n-1}}$ modulo $X^{q^n}-X$, we get \eqref{eq:two_terms_odd} and \eqref{eq:two_terms_even}.
\end{proof}

\section{Equivalence of the Lunardon-Polverino linear sets and their automorphism group}\label{sec:LP}

In this section, we completely determine the equivalence between different members of the Lunardon-Polverino maximum scattered linear sets in a projective line as well as their automorphism groups.
The main results of this section are the following. 

\begin{theorem}\label{th:LP_main_1}
	Let $n,s,t\in \Z$ such that $\gcd(s,n)=\gcd(t,n)=1$, $n\geq 3$ and $1\leq s,t <n/2$. For any nonzero elements $\theta, \delta\in \F_{q^n}$ satisfying $N_{q^n/q}(\theta), N_{q^n/q}(\delta)\neq 1$, let  $f=X^{q^s}+\theta X^{q^{s(n-1)}}$ and $g=X^{q^t}+\delta X^{q^{t(n-1)}}$. 	The linear sets $L_f$ and $L_{g}$ are equivalent  if and only if $s=t$ and one of the following collections of conditions is satisfied:
	\begin{itemize}
		\item[(a)]
		$2\nmid n$, $N_{q^n/q}(\theta)\in \{N_{q^n/q}(\delta^\tau), N_{q^n/q}(1/\delta^\tau)\}$ for some $\tau\in \Aut(\F_{q^n})$
		\item[(b)] 
		$2\mid n$, $N_{q^n/q^2}(\theta)\in \{N_{q^n/q^2}(\delta^\tau),N_{q^n/q^2}(1/\delta^\tau)\}$  for some $\tau\in \Aut(\F_{q^n})$.
	\end{itemize}
\end{theorem}

\begin{theorem}\label{th:LP_main_auto}
		Let $n,s\in \Z$ such that $\gcd(s,n)=1$, $n\geq 3$ and $1\leq s <n/2$. For any nonzero element $\theta\in \F_{q^n}$ satisfying $N_{q^n/q}(\theta)\neq 1$, let  $f=X^{q^s}+\theta X^{q^{s(n-1)}}$. 	Then the $\PGamL$-automorphism group of $L_f$ is
			\[\Aut(L_f)=\begin{cases}
			\mathcal{D}, & n>4;\\
			\mathcal{D}\cup\mathcal{C}, & n=4,
			\end{cases}
			\]
			where
		\begin{equation}\label{eq:cal_D}
			\mathcal{D} :=\left\{
			\left(
			\begin{pmatrix}
			1 & 0 \\ 
			0 & d
			\end{pmatrix} , \tau
			\right)\in \PGamL(2,q^n) : d^{\sigma+1}=\left( \frac{\theta}{\theta^\tau} \right)^\sigma
			\right\}
		\end{equation}
		and 
		\begin{equation}\label{eq:cal_C}
		\mathcal{C} :=\left\{
		\left(
		\begin{pmatrix}
		0 & 1 \\ 
		c & 0
		\end{pmatrix} , \tau
		\right)\in \PGamL(2,q^4) : \left(\frac{c}{\theta^{\tau q}-\theta^{-\tau q^3 }}\right)^{q+1}=\left( -\theta^{\tau q^3+1} \right)^q
		\right\}.
		\end{equation}
		In particular, its size is
			
	\begin{equation}\label{eq:LP_main_auto}
			\#\Aut(L_f) = 
			\begin{cases}
		N_\tau(\theta), & 2\nmid n, 2\mid q;\\
				2N_\tau(\theta), & 2\nmid n, 2\nmid q;\\
				(q+1)N_\tau(\theta), & 2\mid n, n>4;\\
				2(q+1)N_\tau(\theta), & n=4,				
			\end{cases}
		\end{equation}
		where 
		\[
			N_\tau(\theta):=\#\left\{\tau\in \Aut(\F_{q^n}): N_{q^n/q}(\theta) = N_{q^n/q}(\theta^\tau) \text{ or } 1/N_{q^n/q}(\theta^\tau)  \right\}
		\]
		for odd $n$ and 
		\[
		N_\tau(\theta):=\# \left\{\tau\in \Aut(\F_{q^n}): N_{q^n/q^2}(\theta) = \{N_{q^n/q^2}(\theta^\tau) \text{ or }  1/N_{q^n/q^2}(\theta^\tau) \right\}
		\]
		for even $n$.
\end{theorem}

\begin{remark}\label{re:s->n-s}
	Let $n,s\in \Z$ such that $\gcd(s,n)=1$, $s >n/2$. It is easy to see that the linear set $L_f$ defined by $f=X^{q^s}+\theta X^{q^{s(n-1)}}$ is equivalent to $L_g$ where $g=X^{q^{n-s}}+\delta X^{q^{(n-s)(n-1)}}$, with $\delta=\theta^{-1}$. So, the assumptions in Theorem \ref{th:LP_main_1} on $s$ and $t$ are not restrictive.
\end{remark}

In order to get Theorems \ref{th:LP_main_1} and \ref{th:LP_main_auto}, we need to prove a series of lemmas.

\begin{lemma}\label{le:LP_condition_sufficient_1}
	Let $s$ and $n$ be two relatively prime positive integers, and $1\leq s<n/2$. Let $\sigma\colon x \in \fqn \mapsto x^{q^s}\in \fqn$. For any nonzero elements $\theta, \delta\in \F_{q^n}$ satisfying $N_{q^n/q}(\theta), N_{q^n/q}(\delta)\neq 1$, let  $f=X^{\sigma}+\theta X^{\sigma^{n-1}}$ and $g=X^{\sigma}+\delta X^{\sigma^{n-1}}$. There exists $d\in \F_{q^n}$ such that the linear sets $L_f=L_{dg}$  if there exists $\alpha\in \F_{q^n}$ such that
	\begin{equation}\label{eq:LP_condition_d}
		\alpha^{\sigma^2-1} = \left(\frac{\theta}{\delta} \right)^\sigma.
	\end{equation}
\end{lemma}
\begin{proof}
	Replace $X$ by $\alpha X$ in $f(X)/X$. Let $d=\alpha^{\sigma-1}$. Then we obtain
	\[
		\frac{f(\alpha X)}{\alpha X}=\alpha^{\sigma-1} X^{\sigma-1} +\theta \alpha^{\sigma^{n-1}-1}X^{\sigma^{n-1}-1}=d\frac{g(X)}{X},
	\]
	since $(d \delta)^\sigma=\alpha^{\sigma^2-\sigma}\delta^\sigma = \theta^\sigma\alpha^{1-\sigma}$, that is $d\delta=\theta\alpha^{\sigma^{n-1}-1}$. Therefore,  $L_f=L_{dg}$.
\end{proof}

\begin{remark}\label{rm:power_to_norm}
	The sufficient condition given in Lemma \ref{le:LP_condition_sufficient_1} is actually the necessary and sufficient condition for the existences of permutation $q$-polynomials $L_1$ and $L_2$ such that $\cC_g=\{L_1 \circ h \circ L_2(X): h\in \cC_f  \}$, where
	\[ \cC_f := \{aX+bf(X) : a,b\in\F_{q^n} \}. \]
	See \cite[Definition 4.1 and Theorem 4.4]{lunardon_generalized_2018} for more details. By Hilbert's theorem 90, the condition given in \eqref{eq:LP_condition_d} is equivalent to
	\begin{equation}\label{eq:LP_condition_sufficient}
		\begin{cases}
			N_{q^n/q}(\theta)=N_{q^n/q}(\delta), & 2\nmid n;\\
			N_{q^n/q^2}(\theta)=N_{q^n/q^2}(\delta), & 2\mid n.
		\end{cases}
	\end{equation}
\end{remark}

On the other hand, we can get some necessary condition for $L_f=L_{dg}$ for some $d\in \F_{q^n}$ which is slightly different from \eqref{eq:LP_condition_sufficient}.

\begin{lemma}\label{le:LP_condition_NS_d}
	Let $s$ and $n$ be two relatively prime positive integers, and $1\leq s<n/2$. Let $\sigma\colon x \in \fqn \mapsto x^{q^s}\in \fqn$. For any nonzero elements $\theta, \delta\in \F_{q^n}$ satisfying $N_{q^n/q}(\theta), N_{q^n/q}(\delta)\neq 1$, let  $f=X^{\sigma}+\theta X^{\sigma^{n-1}}$ and $g=X^{\sigma}+\delta X^{\sigma^{n-1}}$. There exists $d\in \F_{q^n}$ such that the linear sets $L_f=L_{dg}$  if and only if 
	\begin{equation}\label{eq:LP_condition_NS_d_odd}
		N_{q^n/q}(\theta)\in \{N_{q^n/q}(\delta) , N_{q^n/q}(1/\delta)\},
	\end{equation}
	when $n$ is odd, and 
	\begin{equation}\label{eq:LP_condition_NS_d_even}
	N_{q^n/q^2}(\theta)\in\{N_{q^n/q^2}(\delta) , N_{q^n/q^2}(1/\delta)^\sigma\},
	\end{equation}
	when $n$ is even.
\end{lemma}
\begin{proof}
	First we show that \eqref{eq:LP_condition_NS_d_odd} and \eqref{eq:LP_condition_NS_d_even} are sufficient. As we have proved in Remark \ref{rm:power_to_norm} that \eqref{eq:LP_condition_sufficient} is equivalent to \eqref{eq:LP_condition_d}, if $N_{q^n/q}(\theta)=N_{q^n/q}(\delta) $ for $2\nmid n$ and $N_{q^n/q^2}(\theta)=N_{q^n/q^2}(\delta) $ for $2\mid n$, then there exists $d\in \F_{q^n}$ such that $L_f=L_{dg}$.
	
	 By \eqref{eq:adj}, the adjoint of $f$ is $\hat{f}=\theta^{\sigma} X^\sigma  + X^{\sigma^{n-1}}$.	By Lemma \ref{le:LS_adjoint}, $L_f=L_{\hat{f}}$ and again we may apply Remark \ref{rm:power_to_norm} and Lemma \ref{le:LP_condition_sufficient_1}, obtaining the existence of $d'\in \F_{q^n}$ such that 
	\[
	\left\{ x^{\sigma-1} + \frac{1}{\theta^\sigma} x^{\sigma^{n-1}-1} :x\in \F_{q^n}^* \right\}
	=\left\{ d'\left(x^{\sigma-1} + {\delta} x^{\sigma^{n-1}-1}\right) :x\in \F_{q^n}^* \right\},
	\]
	provided that $N_{q^n/q}(\theta)=N_{q^n/q}(1/\delta)$ for $n$ odd and $N_{q^n/q^2}(\theta)=N_{q^n/q^2}(1/\delta)^\sigma$ for $n$ even. Therefore, we finish the proof of the sufficiency.
	
	Next we prove the necessity of \eqref{eq:LP_condition_NS_d_odd} and \eqref{eq:LP_condition_NS_d_even}.
	By \eqref{eq:jcta_2} for $k=s$, we get
	\[
	\theta^\sigma = d^{\sigma+1} \delta^{\sigma},
	\]
	i.e. 
	\begin{equation}\label{eq:LP_d^sigma+1}
	d^{\sigma+1} = \left(  \frac{\theta}{\delta}\right)^\sigma.
	\end{equation}
	
	When $2\nmid n$, \eqref{eq:LP_d^sigma+1} implies 
	\[ N_{q^n/q}(d)^2=N_{q^n/q}\left(  \frac{\theta}{\delta}\right). \]
	Plugging it into \eqref{eq:two_terms_odd}, we have
	\[1+N_{q^n/q}(\delta)N_{q^n/q}(d)^2 = N_{q^n/q}(d)\left( 1+N_{q^n/q}(\delta) \right),\]
	which is
	\begin{equation}\label{eq:LP_Nd_odd}
		  N_{q^n/q}(d)^2 -\left( 1+\frac{1}{N_{q^n/q}(\delta)} \right)N_{q^n/q}(d) + \frac{1}{N_{q^n/q}(\delta)} =0,
	\end{equation}
	that is $N_{q^n/q}(d)$ is a root of 
	\[ X^2 -\left( 1+\frac{1}{N_{q^n/q}(\delta)} \right)X + \frac{1}{N_{q^n/q}(\delta)}. \]
	It follows that $N_{q^n/q}(d) = 1$ or $N_{q^n/q}(d) =1/N_{q^n/q}(\delta)$ which implies \eqref{eq:LP_condition_NS_d_odd}.
	
	When $2\mid n$, \eqref{eq:LP_d^sigma+1} implies
	\begin{equation}	\label{eq:LP_Nd_even}
		N_{q^n/q}(d)= N_{q^n/q}(d)^\sigma=N_{q^n/q^2}\left(  \frac{\theta}{\delta}\right).
	\end{equation}
	Hence $N_{q^n/q^2}(\theta)=N_{q^n/q}(d)N_{q^n/q^2}(\delta)$ and $N_{q^n/q}(\theta)=N_{q^n/q}(d)^2N_{q^n/q}(\delta)$. Plugging it into  \eqref{eq:two_terms_even}, we have
	\begin{align*}
		&1+N_{q^n/q}(d)^2N_{q^n/q}(\delta)+N_{q^n/q}(d)N_{q^n/q^2}(\delta)+N_{q^n/q}(d)N_{q^n/q^2}(\delta)^\sigma\\ =&N_{q^n/q}(d)(	1+N_{q^n/q}(\delta)+N_{q^n/q^2}(\delta)+N_{q^n/q^2}(\delta)^\sigma),
	\end{align*}
	which is \eqref{eq:LP_Nd_odd} again.	It follows that $N_{q^n/q}(d) = 1$ or $N_{q^n/q}(d)=1/N_{q^n/q}(\delta)$. Together with \eqref{eq:LP_Nd_even}, we get $N_{q^n/q^2}(\theta)=N_{q^n/q^2}(\delta)$ or $N_{q^n/q^2}(\theta)=N_{q^n/q^2}(1/\delta)^\sigma$.
\end{proof}

\begin{remark}\label{rm:LP_n_d_solutions}
   In the proof of Lemma \ref{le:LP_condition_NS_d}, when \eqref{eq:LP_condition_NS_d_odd} or \eqref{eq:LP_condition_NS_d_even} is satisfied, we see that 
\begin{equation}\label{eq:*1}
	\{ d\in \F_{q^n}: L_f=L_{dg} \} = \left\{d\in \F_{q^n}: d^{\sigma+1} = \left( \frac{\theta}{\delta} \right)^\sigma \right\}.
\end{equation}

	Consequently, by Lemma \ref{le:gcd}, there are
\begin{equation}\label{eq:*2}
	\gcd(q^s+1, q^n-1) = \gcd(q+1, q^n-1) = 
\begin{cases}
	1, & 2\nmid n, 2\mid q;\\
	2, & 2\nmid n, 2\nmid q;\\
	q+1, & 2\mid n,
\end{cases}
\end{equation}

	possible choices of $d$ such that $L_f=L_{dg}$.
\end{remark}

\begin{lemma}\label{le:LP_inverse_f}
		Let $n,s,t\in \Z$ such that $\gcd(s,n)=\gcd(t,n)=1$,  and $1\leq s,t <n/2$. For any integer $n>4$ and prime power $q$, define $f=X^{q^s}+\theta X^{q^{n-s}}$ and  $g=X^{q^t}+\delta X^{q^{n-t}}$ with $N_{q^n /q}(\theta), N_{q^n /q}(\delta) \notin \{0,1\}$. If the map $x\mapsto f(x)$ is bijective on $\F_{q^n}$, then  there is no nonzero constant $c$ such that 
	\begin{equation}\label{eq:LP_inverse_f}
	\left\{\frac{cx}{f(x)}:x\in \mathbb{F}_{q^n}^*\right\}=\left\{ \frac{g(x)}{x}:x\in \mathbb{F}_{q^n}^*\right\}.
	\end{equation}
\end{lemma}
\begin{proof}
	Suppose that the inverse map of $f(x)$ is $x\mapsto \sum_{i=0}^{n-1}r_i x^{q^i}$ with $r_i\in \F_{q^n}$. Consequently, 
	\begin{equation*}
	\left(\sum_{i=0}^{n-1}r_i X^{q^i}\right)^{q^s}+\theta\left(\sum_{i=0}^{n-1}r_i X^{q^i}\right)^{q^{n-s}}\equiv X\pmod {X^{q^n}-X}, 
	\end{equation*}
	which means
	\begin{equation}\label{eq:eqinv}
	    \sum_{i=0}^{n-1}\left(r_{i-s}^{q^s}+\theta r_{i+s}^{q^{n-s}}\right)X^{q^i}\equiv X\pmod {X^{q^n}-X},
	\end{equation}
	where the subscript of each $r_i$ is computed modulo $n$.
	Note that
	\begin{equation}\label{eq:ff-1} \left\{\frac{cx}{f(x)}:x\in \mathbb{F}_{q^n}^*\right\}=\left\{\frac{cf^{-1}(x)}{x}:x\in \mathbb{F}_{q^n}^*\right\}, \end{equation}
	since $f$ is invertible.
	Comparing the coefficients of each term in \eqref{eq:eqinv}, we get
	\begin{equation} 
	\label{eq:LP_i_s}
	r_{n-s}^{q^s}+\theta r_{s}^{q^{n-s}}=1,
	\end{equation} 
	and
	\begin{equation} 
	\label{eq:LP_i_2_n-1}
	r_{i-s}^{q^s}=-\theta r_{i+s}^{q^{n-s}}, i\in\{1, 2, 3, ...n-1\}.
	\end{equation}   
	
	Suppose to the contrary that there exists a nonzero constant $c$ such that \eqref{eq:LP_inverse_f} is satisfied. Applying Lemma \ref{le:jcta2018} on \eqref{eq:LP_inverse_f}, taken into account \eqref{eq:ff-1}, we see 
	\[r_0=0.\]
	For $i\in\{1, 2, \cdots,n-1\}$, if $r_{i-s}=0$, then by \eqref{eq:LP_i_2_n-1}  $r_{i+s}=0$. Thus $r_{2s}=0$, because $r_{s-s}=r_0=0$. By induction, 
	\begin{equation}\label{eq:LP_r_2ls=0}
	r_{2\ell s}=0 \text{ for }\ell =0, 1,2,\cdots.
	\end{equation}
	
	If $n$ is odd, then \eqref{eq:LP_r_2ls=0} implies $r_{2\cdot \frac{n-1}{2} s}=r_{n-s}=0$ and $r_{2\cdot \frac{n+1}{2} s}=r_{s}=0$ which contradicts \eqref{eq:LP_i_s}.
	
	If $n$ is even, then \eqref{eq:LP_r_2ls=0} implies $r_{2i}=0$ for all $i=0,1,\cdots, n/2-1$. By \eqref{eq:jcta_2}, 
	\[ r_i r_{n-i}^{q^i} = 0, \]
	for $i\neq t, n-t$, which are equivalent to $r_i=0$ or $r_{n-i}=0$. As $n>4$, there always exists an odd integer $i_0$ satisfying $i_0\neq t, n-t$ and $r_{i_0}=0$. By induction using \eqref{eq:LP_i_2_n-1}, $r_{2\ell s + i_0}=0$ for $\ell =0,1,\cdots$ which implies $r_{s}=r_{n-s}=0$ because $s$ and $n-s$ are both odd. This again contradicts \eqref{eq:LP_i_s}.
\end{proof}

For $n=4$, Lemma \ref{le:LP_inverse_f} does not hold any more.

\begin{lemma}\label{le:LP_inverse_f_n=4}
	For prime power $q$, define $f=X^{q}+\theta X^{q^{3}}$ and  $g=X^{q}+\delta X^{q^{3}}$ over $\F_{q^4}$ with $N_{q^4 /q}(\theta), N_{q^4 /q}(\delta) \notin \{0,1\}$. There exists $c$ such that 
	\begin{equation}\label{eq:LP_inverse_f_n=4}
	\left\{\frac{cx}{f(x)}:x\in \mathbb{F}_{q^4}^*\right\}=\left\{ \frac{g(x)}{x}:x\in \mathbb{F}_{q^4}^*\right\}
	\end{equation}
	if and only if $\theta^{q^2+1}\in \{\delta^{q^2+1 },\delta^{-(q^3+q)}\}$.
	
	When one of the above conditions is satisfied,  the number of $c\in \F_{q^4}$ such that \eqref{eq:LP_inverse_f_n=4} holds is $q+1$.
\end{lemma}
\begin{proof}
	Let $h= X^q - \theta^{-q^3}X^{q^3}$. Then 
	\[h(f(X)) = (\theta^q-\theta^{-q^3})X. \]
Thus $x\mapsto \frac{h(x)}{\theta^q-\theta^{-q^3}}$ is the inverse map of $f(x)$,  and \eqref{eq:LP_inverse_f_n=4} equals
	\[\left\{d\cdot\frac{h(x)}{x}:x\in \mathbb{F}_{q^4}^*\right\}=\left\{ \frac{g(x)}{x}:x\in \mathbb{F}_{q^4}^*\right\}, \]
	where $d=\frac{c}{\theta^q-\theta^{-q^3}}$. 
	
	By Lemma \ref{le:LP_condition_NS_d} for $n=4$, there exists such a $d\in \F_{q^4}$ if and only if 
	$N_{q^4/q^2}(\theta^{-q^3}) = \theta^{-(q^3+q)} = \delta^{q^2+1}$ or $\delta^{-(q^3+q)}$, which is equivalent to  $\theta^{q^2+1}=\delta^{q^2+1 }$ or $\delta^{-(q^3+q)}$.
	
	The last part of the statement of Lemma \ref{le:LP_inverse_f_n=4} follows from Remark \ref{rm:LP_n_d_solutions}.
\end{proof}

\begin{remark}\label{rem:inv}
Comparing the conditions in Lemma \ref{le:LP_inverse_f} and Lemma \ref{le:LP_inverse_f_n=4}, we notice that we have additionally assumed that $f$ is bijective on $\F_{q^n}$ in Lemma \ref{le:LP_inverse_f}. In fact, when $n$ is even, $f$ is always bijective.
Consider $f=X^{q^s}+\theta X^{q^{n-s}}\in \mathcal{L}_{n,q}$ with $\gcd(n, s)=1$ and $N_{q^n /q}(\theta) \notin \{0,1\}$. Suppose that $f$ is not bijective, that is there exists $x \in \fqn^*$ such that
\[ \theta=-x^{q^s-q^{n-s}}. \]
This happens if and only in $N_{q^n/q^e}(\theta)=N_{q^n/q^e}(-1)$, where $e=\gcd(n,n-2s)=\gcd(n,2)$. If $n$ is even then $N_{q^n/q}(\theta)=N_{q^n/q}(-1)=1$, which is a contradiction to the fact that $N_{q^n /q}(\theta) \notin \{0,1\}$.
If $n$ is odd then $e=1$ and $N_{q^n/q}(\theta)=-1$ and in this case $f$ has $q$ roots over $\fqn$.
Therefore, $f$ is bijective if and only if $n$ is even or $n$ is odd and $N_{q^n/q^e}(\theta)\ne-1$.
\end{remark}
Now we are ready to prove Theorem \ref{th:LP_main_1}.

\begin{proof}[Proof of Theorem \ref{th:LP_main_1}]
	Assume that $L_f$ and $L_g$ are equivalent which implies the existence of $M =\begin{pmatrix}
	a & b \\ 
	c & d
	\end{pmatrix} $ over $\F_{q^n}$ and $\tau\in\Aut(\F_{q^n})$ such that
	\begin{equation}\label{eq:LP_main_equivalence}
		\left\{ x^{q^s-1}+\theta x^{q^{n-s}-1} : x\in \F_{q^n}^*  \right\} = \left\{ \frac{cx+d\overline{g}(x)}{ax+b\overline{g}(x)}: x\in \F_{q^n}^*  \right\},
	\end{equation}
	where $\overline{g}(x)=x^{q^t}+\delta^\tau x^{q^{n-t}} $.
	
	Depending on the value of $b$, we separate the rest part into two cases.
	
	\medskip
	
	\noindent\textbf{Case 1:} $b=0$.  As $M$ is viewed as an element in $\PGL(2,q^n)$, we can assume that $a=1$ and \eqref{eq:LP_main_equivalence} becomes
	\begin{equation}\label{eq:LP_main_equi_b=0}
		\left\{ x^{q^s-1}+\theta x^{q^{n-s}-1} : x\in \F_{q^n}^*  \right\} = \left\{ c+d\frac{\overline{g}(x)}{x}: x\in \F_{q^n}^*  \right\}.
	\end{equation}
	By \eqref{eq:jcta_1} in Lemma \ref{le:jcta2018}, we derive $c=0$. By \eqref{eq:jcta_2} for $k=s$, we can get $s=t$; otherwise $\theta^{q^t} = 0$ which contradicts the assumption that $N_{q^n/q}(\theta) \neq 0$. Now \eqref{eq:LP_main_equi_b=0} is exactly $L_f=L_{d \overline{g}}$. By Lemma \ref{le:LP_condition_NS_d}, we derive (a) and (b) of Theorem \ref{th:LP_main_1} as necessary conditions.
	
	\medskip
	
	\noindent\textbf{Case 2:} $b\neq 0$. Without loss of generality, we assume that $b=1$. 
	We are going to prove the following claim:
		
		\medskip
		\textbf{Claim:} $d=0$.
		
Suppose to the contrary that $d\neq 0$. Now
	\[\frac{cx+d\overline{g}(x)}{ax+b\overline{g}(x)}=d+\frac{(c-ad)x}{ax+\overline{g}(x)},\]
	for $x\in\fqn$.
	By definition of the equivalence of linear sets, the map $ x\mapsto ax+\overline{g}(x)$ is invertible (otherwise $\langle (0,1) \rangle_{\F_{q^n}}$ would be in $L_f$) and $c-ad\neq 0$. Let $ h(x)=\sum_{i=0}^{n-1} r_i x^{q^i} $ denote the inverse map of $x\mapsto \frac{ ax+\overline{g}(x)}{c-ad}$. Setting $y=\frac{ ax+\overline{g}(x)}{c-ad}$, we get
	\[
		d+\frac{(c-ad)x}{ax+\overline{g}(x)}=d+\frac{h(y)}{y},
	\]
	for $x\in\fqn$.
	By the definition of $h$ and by comparing the coefficients of $ah(Y)+\overline{g}(h(Y))\equiv (c-ad)Y \pmod{Y^{q^n}-Y}$, we obtain
	\begin{equation}\label{eq:d=0_inverse_0}
		a r_0 +r_{n-t}^{q^t}+\delta^\tau r_t^{q^{n-t}} = c-ad
	\end{equation} 
	and
	\begin{equation}\label{eq:d=0_inverse_i}
		a r_i +r_{i-t}^{q^t}+\delta^\tau r_{i+t}^{q^{n-t}} = 0,
	\end{equation} 	
	for $i\in\{1,2,\cdots,n-1\}$.
	
	Now \eqref{eq:LP_main_equivalence} becomes
	\begin{equation}\label{eq:LP_main_equi_b=1}
	\left\{ x^{q^s-1}+\theta x^{q^{n-s}-1} : x\in \F_{q^n}^*  \right\} = \left\{ d+\sum_{i=0}^{n-1}r_i y^{q^i-1}: y\in \F_{q^n}^*  \right\}.
	\end{equation}
	By \eqref{eq:jcta_1} of Lemma \ref{le:jcta2018},
	\begin{equation}\label{eq:r_0}
		r_0=-d.
	\end{equation}
	Moreover, by \eqref{eq:jcta_2} of Lemma \ref{le:jcta2018},
	\begin{equation}\label{eq:r_kr_(n-k)}
		r_k r_{n-k}^{q^k}=
		\begin{cases}
			\theta^{q^s}, &k=s;\\
			\theta, &k=n-s;\\
			0, &k\neq 0,s,n-s.
		\end{cases}
	\end{equation}
	Depending on the value of $t$, we consider 3 cases. In each case, we will get a contradiction which means the assumption $d\neq 0$ cannot hold.
	
	\medskip
	
	\textbf{Case 2.1:} $s\neq \pm t, \pm 2t, \pm 3t\pmod{n}$. As $\gcd(n,s)=\gcd(n,t)=1$, $n$ must be larger than $4$.
	
	By \eqref{eq:r_kr_(n-k)},
	\[ r_{2t} r_{n-2t}^{q^{2t}}=0,\]
	which means at least one of $r_{2t}$ and $r_{n-2t}$ equals $0$. Without loss of generality, we assume that $r_{n-2t}=0$. 
	By \eqref{eq:d=0_inverse_i} with $i=n-t$, that is
	\[
		a r_{n-t} +r_{n-2t}^{q^t}+\delta^\tau r_{0}^{q^{n-t}} = 0,
	\]
	and $r_0=-d\neq 0$, we get $a\neq 0$ and $r_{n-t}\neq 0$. Together with \eqref{eq:r_kr_(n-k)} for $k=t$, we derive that
	\[r_t=0.\]
	Plugging it into \eqref{eq:d=0_inverse_i} with $i=t$, that is
	\[
	a r_{t} +r_{0}^{q^t}+\delta^\tau r_{2t}^{q^{n-t}} = 0,
	\]
	we obtain $r_{2t}\neq 0$. 	Furthermore, as $r_{2t}\neq 0$ and $r_t=0$,  \eqref{eq:d=0_inverse_i} with $i=2t$, that is
	\[
	a r_{2t} +r_{t}^{q^t}+\delta^\tau r_{3t}^{q^{n-t}} = 0,
	\]
	implies $r_{3t}\neq 0$.
	
	On the other hand, as $r_{n-t}\neq 0$ and $r_{n-2t}=0$,   \eqref{eq:d=0_inverse_i} with $i=n-2t$, i.e., 
	\[
	a r_{n-2t} +r_{n-3t}^{q^t}+\delta^\tau r_{n-t}^{q^{n-t}} = 0,
	\]
	implies $r_{n-3t}\neq 0$. Thus $r_{3t}r_{n-3t}\neq 0$. However, this contradicts \eqref{eq:r_kr_(n-k)}.
	
	\medskip
	\textbf{Case 2.2:} $s\equiv \pm 2t\pmod{n}$. By Remark \ref{re:s->n-s}, we only have to consider the case $s\equiv 2t \pmod{n}$. Since $t\neq \pm s \pmod{n}$,  \eqref{eq:r_kr_(n-k)} with $k=t$ implies that at least one of $r_t$ and $r_{n-t}$ is $0$. Without loss of generality, we assume that $r_t=0$. Moreover, by \eqref{eq:d=0_inverse_0} and $r_0=-d$, 
	\begin{equation}\label{eq:r_n-t}
		r_{n-t}=c^{q^{n-t}}.
	\end{equation}
	By \eqref{eq:d=0_inverse_i} with $i=t$, we get
	\[
		r_0^{q^t}+\delta^\tau r_s^{q^{n-t}}=0,
	\]
	which implies
	\begin{equation}\label{eq:r_s}
		r_s=\frac{d^{q^{2t}}}{\delta^{\tau q^t}}.
	\end{equation}
	By \eqref{eq:r_kr_(n-k)} with $k=n-s$, we have 
	\begin{equation}\label{eq:r_n-s}
		r_{n-s}=\frac{\theta}{r_s^{q^{n-s}}}\neq 0.
	\end{equation}
	Plugging \eqref{eq:r_n-t}, \eqref{eq:r_s} and \eqref{eq:r_n-s}  into the $q^t$-th power of 
	\[a r_{n-t} +r_{n-2t}^{q^t}+\delta^\tau r_{0}^{q^{n-t}} = 0,\]
	we have
	\begin{equation}\label{eq:case2.2-a}
		a^{q^t}c + \delta^{\tau q^t} \left(\frac{\theta}{d}\right)^{q^{2t}} - d\delta^{\tau q^t} = 0. 
	\end{equation}
	
	To show that \eqref{eq:case2.2-a} actually leads to contradiction, we are going to derive an extra condition on the value of $c$ and $a$. 
	
	First, we can simply get $a\neq 0$: if  $a=0$, then by taking $i=3t$ and $i=-3t$ in \eqref{eq:d=0_inverse_i}, we obtain
	\begin{align*}
		r_{2t}^{q^t}+\delta^\tau r_{4t}^{q^{n-t}} &= 0,\\
		r_{n-4t}^{q^t}+\delta^\tau r_{n-2t}^{q^{n-t}} &= 0,
	\end{align*}
	which implies $r_{4t}\neq 0$ and $r_{n-4t}\neq 0$. However, by formula \eqref{eq:r_s}, \eqref{eq:r_n-s} and the assumption $d\neq 0$, it is a contradiction with \eqref{eq:r_kr_(n-k)} for $k=4t$.
	
	As $r_t=0$,   \eqref{eq:d=0_inverse_i} with $i=2t=s$ becomes
	\begin{equation}\label{eq:r_3t}
		a r_{2t} +\delta^\tau r_{3t}^{q^{n-t}} = 0.
	\end{equation}
	By $a\neq 0$, \eqref{eq:r_s} and \eqref{eq:r_3t}, $r_{3t}\neq 0$. 
	By \eqref{eq:r_kr_(n-k)} with $k=3t$, $r_{n-3t}=0$. Hence, \eqref{eq:d=0_inverse_i} with $i=n-2t$ becomes
	\[
		a r_{n-2t} +\delta^\tau r_{n-t}^{q^{n-t}} = 0.
	\]
	Together with \eqref{eq:r_n-s} and \eqref{eq:r_s}, we obtain
	\[
	r_{n-t}^{q^t} = -\left( \frac{a}{\delta^\tau} \right)^{q^{2t}} \frac{\theta^{q^{2t}}}{r_s}=-\left( \frac{a}{\delta^\tau} \right)^{q^{2t}} \frac{\theta^{q^{2t}}}{d^{q^{2t}}} \delta^{\tau q^t}.
	\]
	By \eqref{eq:r_n-t}, we get
	\begin{equation}\label{eq:c}
		c= -\left( \frac{a\theta}{d} \right)^{q^{2t}} \delta^{\tau(q^t-q^{2t})}.
	\end{equation}
	
	Setting $i=3t$ and $n-3t$ in \eqref{eq:d=0_inverse_i}, respectively, we get
	\[
		a r_{3t} +r_{2t}^{q^t}+\delta^\tau r_{4t}^{q^{n-t}} = 0,
	\]
	and
	\[
	a r_{n-3t} +r_{n-4t}^{q^t}+\delta^\tau r_{n-2t}^{q^{n-t}} = 0.
	\]
	Recall that $r_{n-3t}=0$ and $r_{n-s}\neq0$. Consequently, the second equation above implies that $r_{n-4t}\neq 0$. By \eqref{eq:r_kr_(n-k)}, $r_{4t}=0$ and the first equation above implies
	\[
	a r_{3t}=-r_{2t}^{q^t}.
	\]
	Together with \eqref{eq:r_3t}, we have
	\begin{equation}\label{eq:a}
		a^{q^{t}+1} = \delta^{\tau q^t}.
	\end{equation}
	Plugging \eqref{eq:c} and  \eqref{eq:a} into the left-hand side of \eqref{eq:case2.2-a}, 
	\begin{align*}
		&a^{q^t}c + \delta^{\tau q^t} \left(\frac{\theta}{d}\right)^{q^{2t}} - d\delta^{\tau q^t}\\
		=&-\frac{(a^{q^t+1})^{q^t}}{d^{q^{2t}}} \left( \frac{\theta}{\delta^\tau} \right)^{q^{2t}}\delta^{\tau q^t} + \delta^{\tau q^t} \left(\frac{\theta}{d}\right)^{q^{2t}} - d\delta^{\tau q^t}\\
		=&-\frac{\delta^{\tau q^{2t}}}{d^{q^{2t}}} \left( \frac{\theta}{\delta^\tau} \right)^{q^{2t}}\delta^{\tau q^t} + \delta^{\tau q^t} \left(\frac{\theta}{d}\right)^{q^{2t}} - d\delta^{\tau q^t}\\
		=&- d\delta^{\tau q^t}.
	\end{align*}
	Therefore, $- d\delta^{\tau q^t}=0$. However, this is a contradiction with the assumption $d\neq 0$ and $\delta\neq 0$.
	
	\medskip
	\textbf{Case 2.3:}
	$s\equiv \pm 3t\pmod{n}$. Now we assume that $n>4$; otherwise $n=4$ and $s\equiv \pm t\pmod{n}$ which will be investigated later in the next case. By Remark \ref{re:s->n-s}, we only need to consider the case $s\equiv 3t \pmod{n}$. Our strategy is very similar to Case 2.2. As $t\neq \pm s \pmod{n}$,  \eqref{eq:r_kr_(n-k)} with $k=t$ implies that at least one of $r_t$ and $r_{n-t}$ is $0$. Without loss of generality, we assume that $r_t=0$. Moreover, by \eqref{eq:d=0_inverse_0} and $r_0=-d\neq 0$,  we also have
	\begin{equation}\label{eq:s=3t-r_n-t}
	r_{n-t}=c^{q^{n-t}}.
	\end{equation}
	By \eqref{eq:d=0_inverse_i} with $i=t$, we derive 
	\[
	r_{2t} = \frac{d^{q^{2t}}}{\delta^{\tau q^t}}.
	\]
	Together with \eqref{eq:r_kr_(n-k)}, we get $r_{n-2t}=0$. Plugging it into \eqref{eq:d=0_inverse_i} with $i=n-t$, we have $a\neq 0$ and
	\[
	r_{n-t}=\frac{\delta^\tau d^{q^{n-t}}}{a}.
	\]
	From the above equation and \eqref{eq:s=3t-r_n-t}, we derive
	\begin{equation}\label{eq:s=3t_c}
		c=\frac{\delta^{\tau q^t}d}{a^{q^t}}.
	\end{equation}
	Plugging the value of $r_{2t}$, $r_{t}$ and $r_{n-2t}$, $r_{n-t}$ into  \eqref{eq:d=0_inverse_i} for $i=2t$ and $i-n-2t$, respectively, we get
	\begin{equation}\label{eq:s=3t_2t}
		r_{3t}=-\frac{\delta^{\tau(q^t+q^{2t})}}{a^{q^t}d^{q^{3t}}},
	\end{equation}
	and 
	\[
	r_{n-3t}=-\frac{\delta^{\tau(q^{n-t}+q^{n-2t})}}{a^{q^n-2t}} \cdot d^{q^{n-3t}}.
	\]
	Another round of calculation with  \eqref{eq:d=0_inverse_i} for $i=3t$ and $i=n-3t$ leads to
	\begin{equation}\label{eq:s=3t_n-4t_1}
		r_{n-4t} = -(ar_{n-3t})^{q^{n-t}}\neq 0
	\end{equation}
	which by \eqref{eq:r_kr_(n-k)} implies $r_{4t}=0$,	and 
	\[
	r_{3t}=-\frac{d^{q^{3t}}}{a\delta^{\tau q^{2t}}}.
	\]
	The same computation with  \eqref{eq:d=0_inverse_i} for $i=4t$ and $i=n-4t$ provides us
	\[
	r_{5t}\neq 0, ~~r_{n-5t}=0,
	\]
	and 
	\[
		r_{n-4t} = -\frac{\delta^{\tau}}{a} r_{n-3t}^{q^{n-t}}.
	\]
	Together with \eqref{eq:s=3t_n-4t_1}, we get
	\[
	a^{q^t+1}=\delta^{\tau q^t}.
	\]
	However, if we plug it into \eqref{eq:s=3t_c}, we obtain 
	\[ c=ad \]
	which contradicts the assumption that the matrix $M=\begin{pmatrix}
	a&b\\c&d
	\end{pmatrix} $ with $b=1$ is invertible. 
	
	\medskip
	\textbf{Case 2.4:} $t\equiv \pm s \pmod{n}$. Again we only look at the case $t\equiv s \pmod{n}$. By \eqref{eq:r_kr_(n-k)}, $r_t$ and $r_{n-t}$ are both nonzero, and
	\[		r_{2t} r_{n-2t}^{q^{2t}}=0.	\]
	Without loss of generality, we assume that $r_{2t}=0$. By \eqref{eq:d=0_inverse_i} with $i=2t$, that is
	\[
	a r_{2t} +r_{t}^{q^t}+\delta^\tau r_{3t}^{q^{n-t}} = 0,
	\]
	$r_{3t}=-\frac{r_t^{q^{2t}}}{\delta^{\tau q^t}}\neq 0$ whence $r_{n-3t}=0$ by \eqref{eq:r_kr_(n-k)}. Consequently, \eqref{eq:d=0_inverse_i} with $i=n-2t$, that is
	\[
	a r_{n-2t} +r_{n-3t}^{q^t}+\delta^\tau r_{n-t}^{q^{n-t}} = 0,
	\]
	implies 
	\begin{equation}\label{eq:r_n-2t}
		ar_{n-2t}=-\delta^\tau  \cdot r_{n-t}^{q^{n-t}}.
	\end{equation}
	Now let us look at \eqref{eq:d=0_inverse_i} with $i=3t$ and $n-3t$,
	\begin{align*}
		a r_{3t} +r_{2t}^{q^t}+\delta^\tau r_{4t}^{q^{n-t}} &= 0,\\
		a r_{n-3t} +r_{n-4t}^{q^t}+\delta^\tau r_{n-2t}^{q^{n-t}} &= 0.
	\end{align*}
	As $r_{2t}$ is assumed to be $0$ and $r_{3t}\neq 0$, the first equation above implies $r_{4t}\neq 0$. Consequently, \eqref{eq:r_kr_(n-k)} means $r_{n-4t}=0$. Together with $r_{n-3t}=0$ and the second equation above, we get $r_{n-2t}=0$. By \eqref{eq:r_n-2t} and $\delta\neq 0$, we obtain
	\[ r_{n-t}=0. \]
	However, this contradicts \eqref{eq:r_kr_(n-k)} with $k=t=s$ which implies that $r_t$ and $r_{n-t}$ are both nonzero.
	
	Therefore, we have proved that $d=0$.

	As
	\[
	\begin{pmatrix}
	a & 1\\
	c &0
	\end{pmatrix}^{-1}=
	\begin{pmatrix}
	0 & \frac{1}{c}\\
	1 & -\frac{a}{c}
	\end{pmatrix},
	\]
	by considering the inverse map of $M\tau$ and \textbf{Claim}, we get $a=0$. Thus
	\[
	M=\begin{pmatrix}
	0 & 1\\
	c & 0
	\end{pmatrix}.
	\] 
	
	When $n>4$, by applying Lemma \ref{le:LP_inverse_f}, we see that there exists no $c$.	Therefore, the only possibility of $M$ is of the diagonal form satisfying \eqref{eq:LP_main_equi_b=0}.
	
	Let $n=4$, $s=t=1$. By Lemma \ref{le:LP_inverse_f_n=4}, we derive that the existence of such $c\in \F_{q^4}$ is equivalent to require that $N_{q^4/q^2}(\theta)\in \{N_{q^4/q^2}(\delta^\tau),N_{q^4/q^2}(1/\delta^\tau)\}$ given in Theorem \ref{th:LP_main_1} (b).
	
	To summary, we have proved that (a) and (b) of Theorem \ref{th:LP_main_1} are necessary conditions for the equivalence of $L_f$ and $L_g$ with $b=0$ and $b\neq 0$. Therefore, the necessity of (a) and (b) in Theorem \ref{th:LP_main_1} is obtained.
	
	To prove the sufficiency, we only have to  let $s=t$. By Lemma \ref{le:LP_condition_NS_d}, (a) or (b) of Theorem \ref{th:LP_main_1} implies the existence of $d$ such that $L_f=L_{d\overline{g}}$.
\end{proof}

It is similar to prove Theorem \ref{th:LP_main_auto} on the $\PGamL$-automorphism group of a Lunardon-Polverino scattered linear set.

\begin{proof}[Proof of Theorem \ref{th:LP_main_auto}]
	We only have to follow the proof of Theorem \ref{th:LP_main_1} under the assumption that $\theta=\delta$ and $s=t$. 
	
	When $b=0$ in \eqref{eq:LP_main_equivalence}, we always have $a=1$ and $c=0$. 
	Hence, we just need to determine $d\in \F_{q^n}$ and $\tau\in \Aut(\F_{q^n})$ such that $L_{d\overline{f}}=L_f$ where $\overline{f}=X^{q^s}+\theta^\tau X^{q^{n-s}} $. 	By Lemma \ref{le:LP_condition_NS_d} and \eqref{eq:*1}, we get the elements in $\mathcal{D}$;  see \eqref{eq:cal_D}. Moreover, by \eqref{eq:*2}, 
	\[
	|\mathcal{D}|=
	\begin{cases}
		N_\tau(\theta), & 2\nmid n, 2\mid q;\\
		2N_\tau(\theta), & 2\nmid n, 2\nmid q;\\
		(q+1)N_\tau(\theta), & 2\mid n.
	\end{cases}
	\]

	When $b\neq 0$ in \eqref{eq:LP_main_equivalence}, we have $b=1$, $a=0$ and $d=0$.  By Lemma \ref{le:LP_inverse_f}, $n$ must be $4$. For $n=4$, following the proof of Lemma \ref{le:LP_inverse_f_n=4}, we set $\sigma\colon x \in \F_{q^n}\mapsto x^q\in \F_{q^n}$, $h=X^q-\theta^{-\tau q^3}X^{q^3}$, $g=X^q+\theta X^{q^3}$ and $d=\frac{c}{\theta^{\tau q} -\theta^{-\tau q^3}}$. By Remark \ref{rm:LP_n_d_solutions},  $L_g=L_{d h}$ if and only if 
		\[
		d^{q+1} = \left( \frac{\theta}{-\theta^{-\tau q^3}} \right)^q.
		\]
	Hence, we get the elements in $\mathcal{C}$; see \eqref{eq:cal_C}. Furthermore Lemma \ref{le:LP_inverse_f_n=4} and Remark \ref{rm:LP_n_d_solutions}  tell us that 
	\[
	|\mathcal{C}|=
	(q+1)N_\tau(\theta).
	\]
	Therefore, the number of elements in $\Aut(L_f)$ is the same as in \eqref{eq:LP_main_auto}.
\end{proof}

\section{The number of inequivalent Lunardon-Polverino scattered linear sets}\label{sec:number}

In this section, we determine the total number of inequivalent scattered polynomials over $\F_q$ contained in the Lunardon-Polverino construction for given $q$. We always assume that $q=p^r$ where $p$ is a prime and $r$ has the following prime factorization
\[
r=\prod_{i=1}^{\ell}r_i^{s_i}
\]
where $r_1<r_2<\cdots<r_\ell $ and $s_i\geq 1$ for $i=1,\cdots,\ell$.

Define
$$F(r)=\left\{x\in\F_{p^r}:x\notin\F_{p^{r'}} \text{ with } r'<r\text{ and }  r'|r\right \},$$
i.e., $F(r)$ consists of the elements in $\F_{q}$ which do not belong to any proper subfield of $\F_q$. Then by the inclusion-exclusion principle, we get 
\begin{equation*}
|F(r)|=
	p^r-\sum_{i=1}^{\ell}p^{\frac{r}{r_i}}+\sum_{\substack{i,j=1\\i<j}}^{\ell}p^{\frac{r}{r_ir_j}}+\cdots+(-1)^{\ell}p^{\frac{r}{r_1r_2\cdots r_\ell}}.
\end{equation*}

\begin{lemma}\label{le:Kr}
	For $0\leq k< r$, if there exists $x\in F(r)$ such that  $x^{p^k+1}=1$, then $r=1$ or $r$ is even, and 
	\begin{equation}\label{eq:Kr_k}
		k=
		\begin{cases}
		0,& r=1;\\
		\frac{r}{2}, & 2\mid r.
		\end{cases}
	\end{equation}
	Moreover, let $K(r)=\left\{x\in F(r):x^{p^k+1}=1, \text{ for some $k$ with } 0\leq k< r\right \}$. Then
	\begin{equation*}
			|K(r)|=
			\begin{cases}
				\frac{(-1)^{p+1}+3}{2}, & r=1;\\
				p^{\frac{r}{2}}-\frac{(-1)^{p+1}+1}{2}, & r=2^{s_1};\\
				p^{\frac{r}{2}}-\sum_{i=2}^{\ell}p^{\frac{r}{2r_i}}+\sum_{2\leq i<j\leq \ell}p^{\frac{r}{2r_ir_j}}+\cdots+(-1)^{\ell-1}p^{\frac{r}{2r_2\cdots r_\ell}}, & 2\mid r \text{ and } \ell>1;\\
				0, & \text{otherwise}.
			\end{cases}
	\end{equation*}
\end{lemma}
\begin{proof}
	As $x^{p^k+1}=1$, $x^{p^{2k}-1}=1$. Note that $x\in F(r)$, we have $\gcd(2k,r)=r$. If $2\mid r$, then $k$ must be $\frac{r}{2}$ because $0\leq k<r$. If $2 \nmid r$, then we must have $k=0$ and $r=1$.
	
	Next we determine the size of $K(r)$. If $r=1$, then $k=0$ and $x^{p^k+1}=x^2=1$. Thus $|K(r)|=\frac{(-1)^{p+1}+3}{2}$. If $r$ is odd and $r>1$, then by the first part of the lemma, there is no $x\in F(r)$ such that $x^{p^k+1}=1$. Hence $K(r)$ is empty.
	
	Note that $|\{x\in \F_{p^m}: x^{p^k+1}=1  \}| = \gcd(p^k+1, p^m-1)$ which can be determined by \eqref{eq:gcd}. Suppose that $r=2^{s_1}\prod_{i=2}^{\ell}r_i^{s_i}$ with $s_1\geq 1$ which means $r$ is even. By \eqref{eq:Kr_k},  to determine the size of $K(r)$, we only have to consider the number of solutions to $x^{p^k+1}=1$ for $k=\frac{r}{2}$. If $\ell=1$ (that is $r$ is a power of $2$), then
	\begin{align*}
		|K(r)|=& \left|\left\{x\in F(r):x^{p^{r/2}+1}=1\right \}\right|  \\
		=&\gcd(p^\frac{r}{2}+1, p^r-1) -\gcd(p^\frac{r}{2}+1, p^{\frac{r}{2}}-1)\\
		=&p^\frac{r}{2}-\frac{(-1)^{p+1}+1}{2}.
	\end{align*}
	For $\ell>1$, by the inclusion-exclusion principle,
	\begin{align*}
		|K(r)|=& \left|\left\{x\in F(r):x^{p^{r/2}+1}=1\right \}\right|  \\
		=&\gcd(p^\frac{r}{2}+1, p^r-1) - \sum_{i=1}^\ell \gcd(p^\frac{r}{2}+1, p^{\frac{r}{r_i}}-1)  \\
		&+\sum_{1\leq i<j\leq \ell}\gcd(p^\frac{r}{2}+1, p^{\frac{r}{r_ir_j}}-1)+\cdots+(-1)^\ell \gcd(p^\frac{r}{2}+1, p^{\frac{r}{r_1\cdots r_\ell}}-1)\\
		=&p^\frac{r}{2}+1 - \left( \frac{(-1)^{p+1}+3}{2}+ \sum_{i=2}^\ell \left(p^\frac{r}{2r_i }+1\right) \right)
		+\left(  \frac{(-1)^{p+1}+3}{2}(\ell-1) +\sum_{2\leq i<j\leq \ell}\left(p^{\frac{r}{2r_i r_j}}+1\right) \right)+\cdots\\
		&+ (-1)^m \left(
		\frac{(-1)^{p+1}+3}{2} \binom{\ell-1}{m-1} + \sum_{2\leq i_1 <i_2 <\cdots<i_m\leq \ell}\left(p^{\frac{r}{2r_{i_1}r_{i_2}\cdots r_{i_m}}}  +1\right)
		\right)+\cdots\\
		=&p^{\frac{r}{2}}-\sum_{i=2}^{\ell}p^{\frac{r}{2r_i}}+\sum_{2\leq i<j\leq \ell}p^{\frac{r}{2r_ir_j}}+\cdots+(-1)^{\ell-1}p^{\frac{r}{2r_2\cdots r_\ell}}. \qedhere
	\end{align*}
\end{proof}

Denote by $\varphi$ the Euler totient function.

\begin{theorem}\label{th:num_inequi_LP}
	For any integer $n\geq3$ and any prime power $q \neq 2$ with $q=p^r$,  let $\Lambda(n,q)$ denote the total number of the $\PGamL$-inequivalent Lunardon-Polverino scattered linear sets over $\F_{q^n}$. Then
	\[
	\Lambda(n,q)=\left(\sum_{r'\mid r, r'> 1} \frac{|F(r')| + |K(r')|}{2r'} + \epsilon\right) \frac{\varphi(n)}2
	\]
	if $n$ is odd, and
	\[
		\Lambda(n,q)=\left(\sum_{r'\mid r, r'>1} \frac{|F(r')| + |K(r')|}{2r'}+\sum_{r'\mid 2r, r'\nmid r} \frac{|F(r')| - |K(r')|}{2r'} +\epsilon\right) \frac{\varphi(n)}2
	\]
	if $n$ is even, where
	\[
		\epsilon = \begin{cases}
	\frac{p-1}{2},& p\neq 2, 2\nmid n;\\
		\frac{p-3}{2},& p\neq 2, 2\mid n;\\
	0, & p=2.
	\end{cases}
	\]
\end{theorem}
\begin{proof}
	Recall that a Lunardon-Polverino scattered polynomial is defined as $f=X^{q^s}+\theta X^{q^{n-s}}$.
	
	By Theorem \ref{th:LP_main_1}, for different value of $s$, the scattered linear sets are inequivalent. Thus we only have to determine the number of polynomials, for which the associated scattered linear sets are inequivalent,  in $\left\{
		X^{q^s}+\theta X^{q^{n-s}}: N_{q^n/q}(\theta) \neq 0,1
	\right\}$ for given $1\leq s < n/2$.
	
	Let 
	\[
	N(\theta) :=
	\begin{cases}
		N_{q^n/q}(\theta), & 2\nmid n;\\
		N_{q^n/q^2}(\theta), & 2\mid n.
	\end{cases}
	\]
	By Theorem \ref{th:LP_main_1}, $X^{q^s}+\theta X^{q^{n-s}}$ and $X^{q^s}+\delta X^{q^{n-s}}$ are equivalent if and only if 
	\[
	N(\theta)=N(\delta)^{p^k} \text{ or } 1/N(\delta)^{p^k}
	\]
	for some $k\in \{0,\cdots, 2r-1\}$ when $n$ is even, and for some $k\in \{0,\cdots, r-1\}$ when $n$ is odd. 
	
	Therefore, we have only to count the number of the orbits of $\F^*_{q^w}\setminus\{x\in \F_{q^w} : N_{q^w/q}(x)=1\}$ under the action of the group $G$ generated by $\Aut(\F_{q^w})$ and the map $x \in \F_{q^w}^*\mapsto \frac{1}{x}\in \F_{q^w}^*$, where $w=1$ for odd $n$ and $w=2$ for even $n$. Note that 
	\[
	G = \{x\in \F_{q^w}^* \mapsto x^{\pm p^k}\in \F_{q^w}^*: k\in\{0,\cdots, wr-1\} \}.
	\]
	Depending on the parity of $n$, we separate the proof into two cases.
	
	\textbf{Case (I)}. $2\nmid n$.
	Note that $\F_q\setminus \{0,1\}=\left(\bigcup_{r'\mid r, r'\neq 1} F(r') \right)\cup \left(F(1)\setminus \{0,1\}\right)$. When $r'$ is odd, by Lemma \ref{le:Kr}, $|K(r')|=0$ which means there is no $x\in F(r')$ such that $x^{p^k}=1/x$ for any $k$. Consequently, there are exactly $\frac{|F(r')|}{2r'}$ orbits of $F(r')$ under $G$ for $r'>1$. For $r'=1$, it is easy to see that the elements in $F(1)\setminus \{0,1\}$ is partitioned into $\epsilon=\frac{p-1}{2}$ orbits for $p$ odd, and for $p=2$ there is no element in $F(1)\setminus \{0,1\}$.
	
	When $r'$ is even, suppose that there are $N_1$ orbits of length $2r'$ and $N_2$ orbits of length $r'$ of $F(r')$ under $G$. Then
	\[ 2r' N_1 + r' N_2 =|F(r')|.\]
	By the definition of $K(r')$, $N_2=\frac{|K(r')|}{r'}$. By simple computation,
	\[N_1 = \frac{|F(r')| - |K(r')|}{2r'} ~~\text{ and }~~ N_1+N_2 = \frac{|F(r')| + |K(r')|}{2r'}.\]
	
	Since $|K(r')|=0$ for odd $r'>1$, the total number of orbits of elements in $\F_q\setminus\{0,1\}$ are
	\[
	\sum_{r'\mid r, r'\neq 1} \frac{|F(r')| + |K(r')|}{2r'} + \epsilon.
	\]	
	\medskip
	\textbf{Case (II)}. $2\mid n$. Note that
	 \begin{equation*}
	 	\F^*_{q^2}\setminus\{x\in \F_{q^2} : x^{q+1}=1\} 	=\bigcup_{r'\mid 2r} \hat{F}(r')=\left(\bigcup_{r'\mid 2r, r'\neq 1} \hat{F}(r') \right)\cup \left(F(1)\setminus \{0,\pm 1\}\right),
	 \end{equation*}
	 where $\hat{F}(r'):=F(r')\setminus \{x\in F(r'): x^{q+1}\in \{0,1\} \}$.
	 
	 When $r'\mid r$, $x^{q+1}=x^2$. Hence, if $r'>1$, then $\hat{F}(r')=F(r')$.
	 
	 When $r'\mid 2r$ and $r'\nmid r$, $r'$ must be even and $x^{q+1}=x^{p^{r'/2}+1}$. By Lemma \ref{le:Kr}, $\hat{F}(r')=F(r')\setminus  K(r')$.
	 To summarize, for $r'\mid 2r$,
	 \begin{equation}
	 	\hat{F}(r')=\begin{cases}
	 		F(r'), & r'\mid r;\\
	 		F(r')\setminus  K(r'), & r'\nmid r.
	 	\end{cases}
	 \end{equation}
	 
	 By the above analysis and the counting argument in Case (I), when $r'$ is odd and $r'>1$,  there are exactly $\frac{|F(r')|}{2r'}$ orbits of the elements in $\hat{F}(r')$ under $G$.
	 
	 When $r'$ is even, we have to take care of two subcases. If $r' \mid r$, then by the counting argument in \text{Case (I)}, there are exactly $\frac{|F(r')| + |K(r')|}{2r'}$ orbits of the elements in $\hat{F}(r')$ under $G$. If $r'\nmid r$, the number of orbits is $N_1$ which equals $\frac{|F(r')| - |K(r')|}{2r'}$.
	 
	Therefore, the total number of orbits is
	\[
	\sum_{r'\mid r, r'>1} \frac{|F(r')| + |K(r')|}{2r'}+\sum_{r'\mid 2r, r'\nmid r} \frac{|F(r')| - |K(r')|}{2r'} +\epsilon,
	\]
	where $\epsilon$ equals the number of the orbits in $\hat{F}(1)$ under the action of $G$.
\end{proof}

When $r=1$, the value of $\Lambda(n,q)$ is $\epsilon\frac{\varphi(n)}2$ by Theorem \ref{th:num_inequi_LP} (where $\epsilon$ is as in Theorem \ref{th:num_inequi_LP}). However, for $r$ with many divisors, it is in general not trivial to see the explicit value of  $\Lambda(n,q)$. To conclude this section, we provide an upper bound and a lower bound for it and we discuss about its asymptotics.

\begin{theorem}\label{th:bound_Lambda}
	For any integer $n\geq3$ and any prime power $q \neq 2$ with $q=p^r$ and $r>1$,  let $\Lambda(n,q)$ denote the total number of the $\PGamL$-inequivalent Lunardon-Polverino scattered linear sets over $\F_{q^n}$. Let $\sigma(r) $ denote the sum of divisors of $r$, i.e.\ $\sigma(r)=\sum_{d\mid r}d$. When $n$ is odd, 
	\begin{equation}\label{eq:bound_Lambda_odd}
		\frac{p^r-p}{2r}<\frac{\Lambda(n,q)}{\varphi(n)/2}-\epsilon< 
		\begin{cases}
		\frac{p^r}{2r}\left(1+\frac{\sigma (r)-r-1}{p^{r(1-\frac{1}{r_1})}}\right) , &2 \nmid r; \\ 
		\frac{p^r}{2r}\left(1+\frac{\sigma (r)-r-1}{p^{\frac{r}{2}}}
		\right)+\sum_{i=1}^{s_1}p^{2^{i-1}} ,&r=2^{s_1};\\
		\frac{p^r}{2r}\left(1+\frac{\sigma (r)-r}{p^{\frac{r}{2}}}+\frac{\sigma (r)-r-1}{p^{\frac{3r}{4}}}\right) , & 2 \mid r \text{ and } \ell>1.
		\end{cases}
	\end{equation}
	When $n$ is even,
	\begin{equation}\label{eq:bound_Lambda_even_up}
			\frac{\Lambda(n,q)}{\varphi(n)/2}-\epsilon<
		\begin{cases}
		\frac{p^{2r}}{4r}\left(1+\frac{\sigma (2r)-2r-1}{p^{r}}\right) , &2 \nmid r; \\ 
		\frac{p^{2r}}{4r}\left(1+\frac{\sigma (2r)-2r-1}{p^{r}}\right)+\sum_{i=1}^{s_1}p^{2^{i-1}} ,&r=2^{s_1};\\
		\frac{p^{2r}}{4r}\left(1+\frac{\sigma (2r)-2r-1}{p^{r}} + \frac{2}{p^{\frac{3r}{2}}} + \frac{2(\sigma (r)-r-1)}{p^{\frac{7r}{4}}}
		\right) , & 2 \mid r \text{ and } \ell>1.
		\end{cases}
	\end{equation}
	and 
	\begin{equation}\label{eq:bound_Lambda_even_low}
		\frac{\Lambda(n,q)}{\varphi(n)/2}-\epsilon>
	\begin{cases}
	\frac{p^{2r}-p}{4r} -\sum_{i=1}^{s_1+1}p^{2^{i-1}} ,&r=2^{s_1};\\
	\frac{p^{2r}}{4r}\left(1-\frac{1}{p^r}-\frac{\sigma (2r)-2r-1}{p^{\frac{3r}{2}}}-\frac{1}{p^{2r-1}}
	\right)  , &\text{otherwise}.
	\end{cases}
	\end{equation}
\end{theorem}
According to a classical result by Gronwall \cite{Gronwall}, the value of
\[
\limsup_{r\rightarrow \infty} \frac{\sigma(r)}{r \log \log r} =e^\gamma
\]
where $\gamma =0.5772156\dots$ denotes the Euler-Mascheroni constant.
Therefore, by Theorem \ref{th:num_inequi_LP}, for given $n$, the value of $\Lambda(n,q)$ is approximately $\frac{p^r\varphi(n)}{4r}$ for odd $n$ and $\frac{p^{2r}\varphi(n)}{8r}$ for even $n$ provided that $p^r$ is large enough.
\begin{proof}[Proof of Theorem \ref{th:bound_Lambda}]
	 We start by determining the upper bounds. To this aim we first give upper bounds on the quantities appearing in the expression of $\Lambda(n,q)$ as in Theorem \ref{th:num_inequi_LP}.
	 
	 As in the very beginning of Section \ref{sec:number}, we set $r_1$ to be the smallest prime divisor of $r$.
	 As $F(m)\subsetneq \F_{p^m}$ for any $m>1$, 
	\begin{align*}
	\sum_{r'\mid r, r'> 1} \frac{|F(r')|}{2r'}=&	\frac{1}{2r}\sum_{r'\mid r, r'< r}r' \left|F\left(\frac{r}{r'}\right)\right|\\
	<& \frac{1}{2r}\left(p^r+\sum_{r'\mid r,1< r'<r}r' p^{\frac{r}{r'}}\right)\\
	=& \frac{p^r}{2r}\left(1+\sum_{r'\mid r,1< r'<r}\frac{r'}{p^{r-\frac{r}{r'}}}\right)\\
	<&\frac{p^r}{2r}\left(1+\sum_{r'\mid r,1< r'<r}\frac{r'}{p^{r(1-1/r_1 )}}\right).
	\end{align*}
	Thus
	\begin{equation}\label{eq:sup_F}
	\sum_{r'\mid r, r'> 1} \frac{|F(r')|}{2r'} <\frac{p^r}{2r}\left(1+\frac{\sigma (r)-r-1}{p^{r(1-1/r_1)}}\right).
	\end{equation}
	
	Next we derive an upper bound on the value of $\sum_{r'\mid r, r'> 1} \frac{|K(r')|}{2 r'}$. 
	Following the notation and results given in Lemma \ref{le:Kr},  if $2\nmid r$ and $r>1$, then 
	\begin{equation}\label{eq:sup_K_1}
	\sum_{r'\mid r, r'> 1} \frac{|K(r')|}{2 r'}=0.
	\end{equation}
		If $r=2^{s_1}$ and $r>1$,  then 
	\begin{equation}
	\label{eq:sup_K_2}
	\sum_{r'\mid r, r'> 1} \frac{|K(r')|}{2 r'}=\sum_{i=1}^{s_1}\left(p^{2^{i-1}}-\frac{(-1)^{p+1}+1}{2}\right)\leq \sum_{i=1}^{s_1}p^{2^{i-1}}, 
	\end{equation}
	since $|K(r')|\leq p^{\frac{r'}{2}}$. If  $2 \mid r$ and $\ell>1$,   then
	\begin{align*}
	\sum_{r'\mid r, r'> 1} \frac{|K(r')|}{2 r'}=&\frac{1}{2 r}\sum_{r'\mid r, r'< r}r' \left|K\left(\frac{r}{r'}\right)\right|\\
	<&\frac{p^{\frac{r}{2}}}{2 r}\left(1+\sum_{r'\mid r,1< r'<r}\frac{r'}{p^{\frac{r}{2}-\frac{r}{2 r'}}}\right)\\	
	<&\frac{p^{\frac{r}{2}}}{2 r}\left(1+\sum_{r'\mid r,1< r'<r}\frac{r'}{p^{\frac{r}{4}}}\right),
	\end{align*}
	where the first inequality is obtained using that 
	\[ \sum_{r'\mid r, r'< r}r' \left|K\left(\frac{r}{r'}\right)\right| =  \left|K\left(r\right)\right|+ \sum_{r'\mid r, 1<r'< r}r' \left|K\left(\frac{r}{r'}\right)\right|< p^{\frac{r}2}+\sum_{r'\mid r, 1<r'< r}r' p^{\frac{r}{2r'}}.\]
	
	Hence, 
	\begin{equation}\label{eq:sup_K_3}
	\sum_{r'\mid r, r'> 1} \frac{|K(r')|}{2 r'}<\frac{p^{\frac{r}{2}}}{2 r}\left(1+\frac{\sigma (r)-r-1}{p^{\frac{r}{4}}}\right).
	\end{equation}
	
	Now, we are ready to prove the upper bound of $\Lambda(n,q)$.
	
	When  $n$ is odd, 	plugging \eqref{eq:sup_F}, \eqref{eq:sup_K_1}, \eqref{eq:sup_K_2} and \eqref{eq:sup_K_3} into 
	\[\frac{\Lambda(n,q)}{\varphi(n)/2}-\epsilon=\sum_{r'\mid r, r'> 1} \frac{|F(r')| + |K(r')|}{2r'},\]
	 we get the upper bound given in \eqref{eq:bound_Lambda_odd}.

	When $n$ is even, 
	\begin{align*}
	\frac{\Lambda(n,q)}{\varphi(n)/2}-\epsilon=&\sum_{r'\mid r, r'>1} \frac{|F(r')| + |K(r')|}{2r'}+\sum_{r'\mid 2r, r'\nmid r} \frac{|F(r')| - |K(r')|}{2r'}  \\
	<&\sum_{r'\mid 2r, r'>1} \frac{|F(r')| }{2r'}+\sum_{r'\mid r, r'>1} \frac{ |K(r')|}{2r'} .
	\end{align*}
	Plugging \eqref{eq:sup_F}, \eqref{eq:sup_K_1}, \eqref{eq:sup_K_2} and \eqref{eq:sup_K_3} into it, we get \eqref{eq:bound_Lambda_even_up}.
	
	Finally, we compute the lower bound. First, it is easy to see that 
	\[
	\sum_{r'\mid r, r'> 1}|F(r')| =|\F_{p^r}\setminus \F_p|=p^r-p.
	\]
	When $n$ is odd,
	\begin{align*}
	\frac{\Lambda(n,q)}{\varphi(n)/2}-\epsilon>&\sum_{r'\mid r, r'> 1} \frac{|F(r')| + |K(r')|}{2r}  \\
	>&\frac{p^r-p}{2r}  .
	\end{align*}
	Therefore, \eqref{eq:bound_Lambda_odd} is proved.
	
	When $n$ is even,
	\begin{align*}
	\frac{\Lambda(n,q)}{\varphi(n)/2}-\epsilon>&\sum_{r'\mid 2r, r'>1} \frac{|F(r')| }{2r'}-\sum_{r'\mid 2r, r'>1} \frac{|K(r')| }{2r'} \\
	>&\frac{p^{2r}-p}{4r} -\sum_{r'\mid 2r, r'>1} \frac{|K(r')| }{2r'} .
	\end{align*}	
	Then 
	\begin{equation*}
	\frac{\Lambda(n,q)}{\varphi(n)/2}-\epsilon>
	\begin{cases}
	\frac{p^{2r}-p}{4r}  -\sum_{i=1}^{s_1+1}p^{2^{i-1}} ,&r=2^{s_1};\\
	\frac{p^{2r}-p}{4r}  -\frac{p^r}{4r}\left(1+\frac{\sigma (2r)-2r-1}{p^{\frac{r}{2}}}\right) , &\text{otherwise}.
	\end{cases}
	\end{equation*}
	Consequently, we obtain \eqref{eq:bound_Lambda_even_low}.
\end{proof}

\section*{Acknowledgment}
The authors express their gratitude to the anonymous reviewers for constructive comments which are helpful to the improvement of the presentation of this paper.
Wei Tang and Yue Zhou were supported by the Sino-German Mobility Programme M-0157 and the Training Program for Excellent Young  Innovators of Changsha (No.\ kq2106006).
The research of Ferdinando Zullo was supported by the project ``VALERE: VAnviteLli pEr la RicErca" of the University of Campania ``Luigi Vanvitelli'' and was partially supported by the Italian National Group for Algebraic and Geometric Structures and their Applications (GNSAGA - INdAM).

\end{document}